\documentclass[4pt, oneside]{article}
\usepackage{yfonts}
\usepackage{amsmath,amssymb,amsthm}
\usepackage{units,stmaryrd}
\usepackage{graphicx}
\usepackage[all]{xy}
\usepackage{xspace}

\newtheorem{theorem}{Theorem}[section]

\newtheorem{definition}[theorem]{Definition}

\newtheorem{lemma}[theorem]{Lemma}

\newcommand{\RFN}{\ensuremath{{\sf RFN}}\xspace}

\newcommand{\glp}{{\ensuremath{\mathsf{GLP}}}\xspace}

\newcommand{\theory}[1]{\ensuremath{{\mathcal{#1}}}\xspace}

\newcommand{\pra}{\ensuremath{{\mathrm{PRA}}}\xspace}
\newcommand{\zfc}{\ensuremath{{\mathrm{ZFC}}}\xspace}
\newcommand{\pa}{\ensuremath{{\mathrm{PA}}}\xspace}

\newcommand{\isig}[1]{{\ensuremath {\mathrm{I}\Sigma_{#1}}}\xspace}

\newcommand{\ea}{\ensuremath{{\rm{EA}}}\xspace}

\newcommand{\la}{\langle}
\newcommand{\ra}{\rangle}

\begin{document}

\title{$\Pi^0_1$-ordinal analysis beyond first-order arithmetic}
\author{Joost J. Joosten}

\maketitle

\begin{abstract}
In this paper we give an overview of an essential part of a $\Pi^0_1$ ordinal analysis of Peano Arithmetic (\pa) as presented by Beklemishev (\cite{Beklemishev:2004}). This analysis is mainly performed within the polymodal provability logic $\glp_\omega$. 

We reflect on ways of extending this analysis beyond \pa. A main difficulty in this is to find proper generalizations of the so-called Reduction Property. The Reduction Property relates reflection principles to reflection rules. 

In this paper we prove a result that simplifies the reflection rules. Moreover, we see that for an ordinal analysis the full Reduction Property is not needed. This latter observation is also seen to open up ways for applications of ordinal analysis relative to some strong base theory.
\end{abstract}

\maketitle

\section{Introduction}

Primitive Recursive Arithmetic (\pra) is a rather weak formal theory about the natural numbers which is among various philosophers, logicians and mathematicians held to be a good candidate for the concept of finitism (see e.g., \cite{Tait81}).
The concept of finitism tries to capture those mathematical truths and that part of mathematical reasoning which is true beyond doubt and which not uses strong assumptions on infinite mathematical entities.

Gentzen showed in his seminal paper from 1936 (\cite{Gentzen:1936}) that \pra together with some clearly non-finitist notion of  transfinite induction for easy formulas along a rather small ordinal could prove the consistency of Peano Arithmetic (\pa). 

This result can be seen as a partial realization of Hilbert's programme where finitist theories are to prove the consistency of strong mathematical theories. Of course, since G\"odel's incompleteness results we know that this program is not viable but Gentzen's consistency proof seems to clearly isolate the non-finitist part needed for such a consistency proof.

Since Gentzen's consistency proof, the scientific community has tried to calibrate the proof-strength of various theories other than \pa. The amount of transfinite induction needed in these consistency proofs is referred to as the proof-theoretic ordinal of a theory. There are various ways of defining and computing these ordinals and for most natural theories these methods all yield the same ordinals. 

Among these methods the more novel one was introduced by Beklemishev \cite{Beklemishev:2004} and is based on modal provability logics. The corresponding ordinals are referred to as the $\Pi_1^0$ ordinals. In this paper we shall sketch the method for computing these $\Pi^0_1$ ordinals. So far, the method has only been applied successfully to theories as \pa and its kin. In this paper we reflect on ways of extending this analysis beyond \pa. 

A main difficulty in this is to find proper generalizations of the so-called Reduction Property. The Reduction Property relates reflection principles to reflection rules. In this paper we prove a result that simplifies the reflection rules. Moreover, we see that for an ordinal analysis the full Reduction Property is not needed. This latter observation is also seen to open up ways for applications of ordinal analysis relative to some strong base theory.

Before we can start looking into the ordinal analysis, we must first introduce some basic knowledge concerning arithmetic and provability logics.

\section{Prerequisites}

All results in this section are given without proofs. For further background the reader is referred to standard textbooks like \cite{Boo93} or \cite{HP}.

\subsection{Arithmetic}
By the language of arithmetic we understand in this paper the language based on the symbols $\{ 0, S, +, \times, \exp, \leq, = \}$ where $\exp$ denotes the function $x\mapsto 2^x$.

Formulas of arithmetic are stratified in complexity classes as usual. Thus, $\Delta_0^0$ formulas are first-order formulas where all quantifiers refer to numbers and are bounded by some term $t$ as in $\forall \, x{\leq}t$ where of course $x\notin t$. 

We define $\Sigma^0_0 := \Pi^0_0 := \Delta_0^0$. If $\varphi(\vec x, \vec{y}) \in \Sigma^0_n$, then $\forall \vec x \ \varphi(\vec x, \vec{y}) \in \Pi^0_{n+1}$ and likewise, if $\varphi(\vec x, \vec{y}) \in \Pi^0_n$, then $\exists \vec x \ \varphi(\vec x, \vec{y}) \in \Sigma^0_{n+1}$. 

Similarly we define the hierarchies $\Pi^n_m$ where now the number of $n$th-order quantifiers is counted although in this paper we shall at most need second order quantifiers. 

By \ea we denote the arithmetic theory of \emph{Elementary Arithmetic}. This theory is formulated in the language of arithmetic. Apart from the defining axioms for the symbols in the language, \ea has an induction axiom $I_\varphi$ for each $\Delta_0^0$ formula $\varphi(\vec x, y)$ (that may contain $\exp$): 
\[
I_\varphi(\vec x) : \ \ \varphi(\vec x, 0) \wedge \forall\, y\ (\varphi(\vec x, y) \to \varphi(\vec x, y+1)) \ \to \ 
\forall y \varphi(\vec x, y).
\]

By $\ea^+$ we denote $\ea$ plus the axiom that states that super-exponentiation --the function that maps $x$ to the $x$ times iteration of $\exp$-- is a total function.

By $\isig{n}$ we denote the theory that is as \ea except that it now has induction axioms $I_\varphi$ for all formulas $\varphi (\vec x) \in \Sigma_n^0$. The theory \pa is the union of all the $\isig{n}$ in that it has induction axioms for all arithmetic formulas.

\subsection{Transfinite induction}
Greek letters will often denote ordinals and as usual we denote by $\varepsilon_0$ the supremum of $\{\omega, \omega^\omega, \omega^{\omega^\omega}, \ldots \}$. Apart from considering induction along the natural numbers we shall consider induction along transfinite orderings too. If $\la\Gamma,\prec\ra$ is a natural arithmetical representation in \ea of some ordinal we denote by ${\sf TI}[X, \Gamma]$ the collection of transfinite induction axioms for all formulas in $X$:
\[
\forall y \ \big( \forall \, y'{\prec} y\  \varphi(\vec x, y') \to  \varphi(\vec x, y)\big) \ \to \ \forall y\  \varphi(\vec x, y)
\ \ \ \ \ \mbox{ with $ \varphi(\vec x, y) \in X$.}
\]

\subsection{Formalized metamathematics}

Throughout this paper we shall use representations in arithmetic of various metamathematical notions. In particular we fix some G\"odel numbering to represent formulas and other syntactical objects in arithmetic. 

Moreover, we assume that we can represent r.e. theories in a suitable way so that we can speak of ``the formula $\varphi$ is provable in the theory $T$'' whose formalization we shall denote by $\Box_T\varphi$. Dually, we shall use the notion of ``the formula $\varphi$ is consistent with the theory $T$" which is denoted by ${\sf Con}_T(\varphi)$ or $\Diamond_T \varphi$.

If we write $\Box_T\varphi(\dot x)$ we denote by that a formula whose free variable is $x$, and so that provably for every $x$, the formula $\Box_T\varphi(\dot x)$ is equivalent to $\Box_T\varphi(\overline x)$. Here $\overline x$ denotes the numeral of $x$, that is, 
\[
\overline{x} = \overbrace{S\ldots S}^{x \mbox{ times}}0.
\]

\section{Provability logics}

The logics $\glp_\Lambda$ provide provability logics for a series of provability predicates/modalities $[\alpha]$ of increasing strength.

\begin{definition}
Let $\Lambda$ be an ordinal. By $\glp_\Lambda$ we denote the poly-modal propositional logic that has for each $\alpha < \Lambda$ a modality $[\alpha]$ (that syntactically binds as the negation symbol). The axioms of \glp are all propositional logical tautologies in this signature together with instantiations of the following schemes:
\[
\begin{array}{ll}
[\alpha ] (A \to B )  \to  ([\alpha ]A \to [\alpha] B)& \forall \alpha < \Lambda; \\
{[}  \alpha ] ([\alpha ]A \to A )  \to  [\alpha ]A & \forall \alpha < \Lambda; \\
{[}\alpha ] A  \to  [\beta]A  & \forall \alpha \leq \beta < \Lambda; \\
\la \alpha \ra A  \to  [\beta] \la \alpha \ra A & \forall \alpha < \beta < \Lambda. \\
\end{array}
\]
As always we have that $\la \alpha\ra A := \neg [\alpha ] \neg A$. The rules are Modus Ponens and a Necessitation rule for each modality below $\Gamma$, that is, $\frac{A}{[\alpha]A}$. 
\end{definition}

By $\glp$ we shall denote class-size logic which is the ``union'' of $\glp_\Lambda$ over all $\Lambda \in {\sf On}$. The closed fragment of $\glp_\Lambda$ is the set of its theorems that do not contain any propositional variables and is denoted by $\glp_\Lambda^0$. It turns out that $\glp_\Lambda^0$ is already a very rich structure that is strong enough to perform major parts of our ordinal analysis. Some privileged inhabitants of $\glp_\Lambda^0$ are the so-called \emph{worms}. They are just iterated consistency statements in $\glp_\Lambda^0$.

\begin{definition}[$S^\Lambda$]
$\top \in S^\Lambda$, and if both $A \in S^\Lambda$ and $\beta < \Lambda$, then $\la \beta \ra A \in S^\Lambda$.
\end{definition}
We can define an order $<_0$ on $S^\Lambda$ by $A <_0 B \ :\Leftrightarrow \ \glp_\Lambda^0 \vdash B\to \la 0 \ra A$. It is known (\cite{Beklemishev:2005, BeklemshevFernandezJoosten2011}) that this ordering makes $S^\Lambda$ into a well-order.

\subsection{The Reduction Property}

Japaridze (\cite{Japaridze:1988}) has shown $\glp_\omega$ to be arithmetically sound and complete if we interpret $[n]$ as ``provable by $n$ applications of the $\omega$-rule''. Ignatiev then showed in \cite{Ignatiev:1993} that this completeness result actually holds for a wide range of arithmetical readings of $[n]$. In particular, we still have completeness when reading $[n]$ as a natural formalization of ``provable in \ea together with all true $\Pi^0_n$ sentences''.

For the remainder of the section, let $[n]$ refer to this latter reading. The advantage of this reading is that certain worms can be easily linked to reflection principles and fragments of arithmetic: 

\begin{lemma}\label{theorem:ConsistencyEquivalentToReflectionEquivalentToInduction}
$\ea + \la n+2 \ra \top \ \equiv \ \ea + {\sf RFN}_{\Sigma_{n+2}}(\ea) \ \equiv \ \isig{n+1}$.
\end{lemma}

\begin{proof}
We shall refrain from distinguishing a modal formula from its arithmetical interpretation if the context allows us to. Thus, in this statement, $\la n+2 \ra \top$ clearly refers to the formalized statement that $\ea$ together with all true $\Pi_{n+2}^0$-formulas is consistent. 

By ${\sf RFN}_{\Sigma_{n+1}}(\ea)$ we denote the set of axioms $\{ [0]_\ea \sigma(\dot x) \to \sigma(x) \mid \sigma \in \Sigma_{n+1} \}$.

The $\ea + \la n+2 \ra \top  \equiv \ea + {\sf RFN}_{\Sigma_{n+2}}(\ea)$ equivalence is actually rather easy and can be found in \cite{BeklemishevSurvey:2005}. The remaining  equivalence $\ea + {\sf RFN}_{\Sigma_{n+2}}(\ea) \equiv \isig{n+1}$ is a classical result by Leivant \cite{Leivant:1983}.
\end{proof}

We can write ${\sf RFN}_{\Sigma_n}(EA)$ also as $\pi(x) \to \Diamond_\ea \pi (\dot x)$ for $\pi(x) \in \Pi_n^0$. This in turn can be studied as a rule rather than an implication: $\frac{\pi(x)}{\Diamond_\ea \pi (\dot x)}$. In this rule we can vary both the complexity class to which $\pi(x)$ belongs and the notion of provability used (here just $\Diamond_\ea$ which is $\la 0 \ra_\ea$) giving rise to a scala of different rules. In \cite{Beklemishev:2004, BeklemishevSurvey:2005} these rules are introduced and studied.

\begin{definition}
The Reflection rule $\Pi_m^0{-}{\sf RR}^n(\mathcal{U})$ is defined as $\frac{\pi(x)}{\la n \ra_{\mathcal{U}}\pi(\dot x)}$ where $\pi \in \Pi_m^0$.
\end{definition}
 
The following theorem is called the Reduction Property. A proof of it can be found in either one of \cite{Beklemishev:2004, BeklemishevSurvey:2005}.

\begin{theorem}\label{theorem:ReductionProperty}
The theory $\ea + {\sf RFN}_{\Sigma_{n+1}}$ is $\Pi_{n+1}^0$-conservative over $\ea + \Pi^0_{n+1}{-}{\sf RR}^n(\ea)$.
\end{theorem}

The Reduction Property can be stated and proved under more general conditions but for the current purpose this presentation suffices. At first glance it might seem a mere technicality but it implies various classical results like Parson's result that $\isig{1}$ is $\Pi^0_2$-conservative over \pra. Moreover, as we shall see, it is one of the main ingredients in our ordinal analysis.

\subsection{Simplifying the Reflection Rule}

In this subsection we shall see that we can simplify the family of reflection rules considerably. We prefer to work in a general setting here. Thus, let $[n]_{\mathcal{U}}$ be any series of provability predicates over a theory $\mathcal{U}$ that is sound for \glp.
Moreover, we have for each $n\in \omega$ that the formalized deduction theorem holds:
\[
\mathcal{U} \vdash [n]_{\mathcal{U} + \varphi} \psi \ \ \Longleftrightarrow \ \ \mathcal{U} \vdash [n]_{\mathcal{U} }( \varphi\to \psi).
\]

The Reflection Rule as studied in the \glp project has currently two parameters $n$ and $m$:

\[
\Pi_m^0{\sf - RR}^n(U+\varphi) \ := \ \frac{\psi}{\langle n\rangle(\varphi \wedge \psi)} \ \ \ \mbox{\ for $\varphi \in \Pi_m$.}
\]
In virtue of the easy lemmas below we shall see that we can drop the parameter $m$ as over $\mathcal{U}$, for $m>n$ all the versions turn out to be equivalent, and for $m\leq n$ the rule is just equivalent to the axiom $\langle n \rangle \varphi$. Thus, we propose to just speak of the ${\sf RR}^n(U+\varphi)$:
\[
{\sf RR}^n(U+\varphi) \ \ := \ \ \frac{\psi}{\langle n \rangle (\psi \wedge \varphi)}
\]
without any restriction on the complexity of $\psi$. In the remainder of this subsection, we shall assume that $\la n \ra \varphi$ is of complexity $\Pi_{n+1}^0$. However if this were not the case, the arguments go through exactly the same by replacing each occurrence of $\Pi^0_{n+1}$ by
$\widetilde {\Pi^0_{n+1}}$ where $\widetilde {\Pi^0_{n+1}}$ represents some natural complexity class to which $\la n \ra \varphi$ belongs.

\begin{definition}\label{definition:FundamentalSequence}
Let $Q^0_n(\varphi) = \langle n \rangle_{\mathcal{U}} \varphi$ and $Q^{k+1}_n(\varphi) = \langle n \rangle_{\mathcal{U}} (\varphi \wedge Q^k_n(\varphi))$.
\end{definition}

\begin{lemma}\label{theorem:ReflectionRuleSimplification}
Let $l,m,n \in \omega$ and $l >n<m$. We have that
\[
\mathcal{U} + \Pi_l-{\sf RR}^n(\mathcal{U} + \varphi) \ \equiv \ \mathcal{U} + \Pi_m-{\sf RR}^n(\mathcal{U} + \varphi) \ \equiv \ \mathcal{U} +
\{  Q^k_n (\varphi) \mid k\in \omega  \}.
\]
\end{lemma}

\begin{proof}
As the complexity of $Q_n^k(\varphi)$ is $\Pi_{n+1}$ for any $k$ and $\varphi$, it is easy to see by an induction on $k$ that for any $k,m,n\in \omega$ we have 
\[
\Pi_m-{\sf RR}^n(\mathcal{U} + \varphi) \ \vdash  Q^k_n (\varphi)
\]
so that 
$\mathcal{U} + \Pi_m-{\sf RR}^n(\mathcal{U} + \varphi) \ \supseteq \ \mathcal{U} +
\{  Q^k_n (\varphi) \mid k\in \omega  \}$.

For the reverse inclusion we do induction on the number of applications of the rule $\Pi_m-{\sf RR}^n(\mathcal{U} + \varphi)$. So, suppose that for some $\chi \in \Pi_m$ we have that $\mathcal{U} + \Pi_m-{\sf RR}^n(\mathcal{U} + \varphi) \vdash \chi$. By the IH we have $\mathcal{U} \vdash Q^k_n(\varphi) \to \chi$ for some natural number $k$.
But then by necessitation we have $\mathcal{U} \vdash [n]( Q^k_n(\varphi) \to \chi)$ whence
$\mathcal{U} + Q^{k+1}_n(\varphi) \vdash \langle n \rangle_{\mathcal{U}}( \varphi \wedge \chi)$ as was to be shown.

\end{proof}

\begin{lemma}
Let $m,n \in \omega$ with $n\geq m$. We have that
\[
\mathcal{U} + \Pi_m-{\sf RR}^n(\mathcal{U} + \varphi) \ \equiv \ \mathcal{U} + \la n \ra \varphi.
\]
\end{lemma}

\begin{proof}
Clearly, by one application of the $\Pi_m-{\sf RR}^n(\mathcal{U} + \varphi)$ rule we obtain $\frac{\top}{\la n\ra  \varphi}$. Thus 
\[
\mathcal{U} + \la n \ra \varphi \subseteq \mathcal{U} + \Pi_m-{\sf RR}^n(\mathcal{U} + \varphi).
\]
To prove the converse implication we show that $\mathcal{U} + \la n \ra \varphi$ is closed under the rule. Thus, reason in  $\mathcal{U} + \la n \ra \varphi$ and suppose we have proved $\psi$ with $\psi \in \Pi_m$. As $\psi \in \Sigma_{n+1}$ we have that $\psi \to [n] \psi$. We combine this with $\la n\ra \varphi$ to obtain the required $\la n\ra (\psi \wedge \varphi)$.
\end{proof}

We note that a similar argument applies to $\glp_\Lambda$ once we have fixed suitable formulas $Q^k_\alpha(\varphi)$ there and have specified complexity classes for formulas of the form $\la \alpha \ra \psi$.

\subsection{The Reduction Property revisited}

In more generality, we can define for \glp formulas --not just worms-- an ordering over \glp: 
\[
\varphi <_\alpha \psi \ \:\Leftrightarrow \ \glp \vdash \psi \to \la \alpha \ra  \varphi.
\]
With respect to these orderings, consistency statements behave very well and admit some sort of fundamental sequence. For any formula $\varphi$ we defined $Q^k_\alpha (\varphi)$ for $k\in \omega$ by $Q^0_\alpha (\varphi) := \la \alpha \ra \varphi$ and $Q^{k+1}_\alpha (\varphi) := \la \alpha \ra ( \varphi \wedge Q^k_\alpha (\varphi))$. With these formulas at hand we can state part of the fundamental sequence result to the effect that the formulas $\{ Q^k_n(\varphi)\}_{k\in \omega}$ substitutes a fundamental sequence of $\la n+1 \ra \varphi$.

\begin{lemma}\label{theorem:ReductionPropertyInclusion}
For each $k \in \omega$ we have that $\glp \vdash \la \alpha +1 \ra \varphi \to Q^k_\alpha(\varphi)$ whence also
$\glp \vdash \la \alpha +1 \ra \varphi \to \la \alpha \ra Q^k_\alpha(\varphi)$.
\end{lemma}
A proof of this lemma is not hard and can be found, e.g., in \cite{BeklemishevSurvey:2005}. The other half of the fundamental sequence result is in virtue of the above just recasting the Reduction Property in terms of \glp.

\begin{theorem}\label{theorem:ReductionPropertyInGLP}
$\ea + \la n+1 \ra \varphi$ is $\Pi_{n+1}$-conservative over $\ea + \{ Q^k_n(\varphi) \mid k\in \omega\}$.
\end{theorem} 

\begin{proof}
By Lemma \ref{theorem:ReductionPropertyInclusion} we see that $\ea + \{ Q^k_n(\varphi) \mid k\in \omega\} \ \subseteq \ \ea + \la n+1 \ra \varphi$. The $\Pi_{n+1}$-conservativity follows directly from the Reduction Property --Theorem \ref{theorem:ReductionProperty}-- and Lemma \ref{theorem:ReflectionRuleSimplification} above.
\end{proof}

The main ingredient of the proof of the Reduction Property is a cut-elimination argument. Thus, as was noted in previous papers, the theorem above --Theorem \ref{theorem:ReductionPropertyInGLP}-- is formalizable as soon as the superexponential function is provably total and in particular in $\ea^+$. From this fact we get a powerful result concerning provable equi-consistency (see e.g. \cite{BeklemishevSurvey:2005}):

\begin{theorem}\label{theorem:equiconsistencyReductionProperty}
For $m\leq n$ we have that $\ea^+ \vdash \la m\ra \la n+1 \ra \varphi \ \leftrightarrow \ \forall k\ \la m\ra Q^k_n (\varphi)$.
\end{theorem}

\begin{proof}
We reason in $\ea^+$ and prove the equivalence by contraposition. Lemma \ref{theorem:ReductionPropertyInclusion} is actually already provable in \ea so that we see 
\[
\exists k\ [m]_{Q^k_n(\varphi)}\bot \to [m]_{\la n+1 \ra \varphi} \bot .
\]
For the other direction we invoke the Reduction Property as stated in Theorem \ref{theorem:ReductionPropertyInGLP}. 

So, still reasoning in $\ea^+$, we suppose that $[m]_{\la n+1 \ra \varphi} \bot$. Let $\pi$ be the conjunction of $\Pi_m^0$ sentences that are used in the $\ea + \la n+1\ra \varphi$ proof of $\bot$. Thus, we get that $[0]_{\la n+1\ra \varphi} \neg \pi$. As $\neg \pi \in \Pi_{n+1}^0$ we get by the formalized reduction property that $[0]_{Q^k_n(\varphi)}\neg \pi$ for some (possibly non-standard) number $k$. The latter implies $[m]_{Q^k_n(\varphi)}\bot$ and we are done.
\end{proof}

\section{A $\Pi_1^0$-ordinal analysis for \pa}

The following theorem with proof can be found in full detail in \cite{BeklemishevSurvey:2005}. We present here the main part of the proof but refer for certain claims made here to \cite{BeklemishevSurvey:2005}.

\begin{theorem}\label{theorem:PaIsConsistent}
$\ea^+ + {\sf TI}[\Pi^0_1, \varepsilon_0] \vdash {\sf Con}(\pa)$
\end{theorem}

\begin{proof}
It is well-known that the equivalence between reflection, induction and consistency as stated in Lemma \ref{theorem:ConsistencyEquivalentToReflectionEquivalentToInduction} can actually be formalized in $\ea^+$. Thus, we reason in $\ea^+$ and observe that we have $\pa \subseteq \ea + \{  \la 1 \ra \top, \la 2\ra \top, \la 3 \ra \top, \la 4 \ra \top, \ldots \}$. Consequently, ${\sf Con}(\ea + \{  \la 1 \ra \top, \la 2\ra \top, \la 3 \ra \top, \la 4 \ra \top, \ldots \} )\to {\sf Con} (\pa)$ and we shall complete our proof by showing ${\sf Con}(\ea +  \{  \la 1 \ra \top, \la 2\ra \top, \la 3 \ra \top, \la 4 \ra \top, \ldots \})$. For this, it suffices to show
\begin{equation}\label{equation:ConsistencyOfConsistencies}
\forall n\  \la 0 \ra \la n \ra \top.
\end{equation}

We shall prove this by transfinite induction.
It is known that $\la S^\omega,<_0 \ra$ is provably in \ea isomorphic to $\la \varepsilon_0 , <\ra$.
Thus it suffices to perform a transfinite induction over the structure $\la S^\omega, <_0 \ra$. Clearly 
\begin{equation}\label{equation:ConsistencyOfWorms}
\forall \, A{\in}S^\omega\ \la 0\ra A
\end{equation}
implies \eqref{equation:ConsistencyOfConsistencies}, so we shall prove \eqref{equation:ConsistencyOfWorms} by transfinite induction over $\la S^\omega, <_0 \ra$. We set out to prove $\forall \, A{\in } S^\omega\ (\forall \, A'{<_0}A \ \la 0 \ra A' \to \la 0 \ra A )$ from which \eqref{equation:ConsistencyOfWorms} follows, and distinguish three cases:
\begin{enumerate}
\item
$A = \top$ in which case we have $\la 0 \ra \top$ as $\ea^+$ proves the consistency of \ea.

\item 
$A$ is of the form $\la 0 \ra B$ for some worm $B$.

It is well-known that $\ea^+ \vdash {\sf RFN}_{\Sigma_1^0}(\ea)$. So in particular, as $[0] B$ is a $\Sigma_1^0$-sentence, we get $[0] [0] B \to [0] B$. Thus also $\la 0 \ra B \to \la 0 \ra \la 0\ra B$. However, as $B<_0 A$ we have by the induction hypothesis that $\la 0 \ra B$ and we are done.

\item
$A$ is of the form $\la n+1 \ra B$ for some worm $B$ and natural number $n$.

So, we need to prove $\la 0\ra \la n+1 \ra B$. By Theorem \ref{theorem:equiconsistencyReductionProperty} we get that 
\[
\la 0\ra \la n+1 \ra B \leftrightarrow \forall k \ \la 0 \ra Q^k_n(B).
\]
However, as for each $k\in \omega$ we have by Lemma \ref{theorem:ReductionPropertyInclusion} that $Q^k_n(B) <_0 \la n+1\ra B$ we are done by the induction hypothesis.
\end{enumerate}
\end{proof}

On the basis of Theorem \ref{theorem:PaIsConsistent} one could decide to call $\varepsilon_0$ the proof-theoretical ordinal of \pa. Like many other ordinal analyses, the current analysis is susceptible to plugging in pathological ordinal notation systems so as to get way weaker or stronger proof-theoretical ordinals for \pa. However, we feel confident to judge ourselves which notation system is natural enough to use and which not.

We shall now briefly say why this particular ordinal is called the $\Pi_1^0$ ordinal of \pa. If we define Turing progressions $\ea^\alpha$ of \ea by transfinite induction in the standard way as $\ea^0 := \ea$, and $\ea^\alpha := \cup_{\beta<\alpha}(\ea_\beta + {\sf Con}(\ea_\beta))$ we can define a $\Pi^0_1$ proof theoretical ordinal based on these $\ea^\alpha$. For a target theory $T$ we define $|T|_{\Pi^0_1}$ --the $\Pi^0_1$ proof theoretical ordinal of $T$-- to be the smallest $\alpha$ for which $\ea^\alpha$ comprises all the $\Pi^0_1$ consequences of $T$.

For natural theories $T$ and natural ordinal notation systems, this ordinal will coincide with the ordinal obtained by an analysis presented in Theorem \ref{theorem:PaIsConsistent}. Moreover for $T=\pa$ and various sub-systems $T$ of \pa, it is known that $|T|_{\Pi^0_1}$ coincides with all the other known ordinal analyses like $|T|_{\Pi^1_1}$ or $|T|_{\Pi^0_2}$. 

We mention these other proof-theoretical ordinals here without further detail and just to provide some context. In this same spirit it is worth mentioning that $|T|_{\Pi^0_1}$ is more fine-grained than any of the others. For example, $|\pa + {\sf Con}(\pa)|_{\Pi^1_1} = |\pa + {\sf Con}(\pa)|_{\Pi^0_2} = |\pa|_{\Pi^0_2}= \varepsilon_0$ whereas $|\pa + {\sf Con}(\pa)|_{\Pi_1^0} = \varepsilon_0 \cdot 2$.

\section{Ingredients for going beyond \pa}

The paradigm for $\Pi^0_1$ is nice in that it provides a more fine-grained analysis than all other ordinal analyses around. In a sense, it provides the finest analysis possible as different true theories will at least differ on $\Pi^0_1$ sentences. A critique to the paradigm is that the analysis has so far only been performed for rather weak mathematical theories: \pa and its kin.

If we wish to address stronger theories than \pa there are two paths that one can take. In the next subsection we discuss one such path where the base theory is strengthened. In the remaining subsection we speak about the approach where we strengthen $\glp_\omega$ to $\glp_\Lambda$ with $\Lambda > \omega$. 

\subsection{Relative $\Pi^0_1$ ordinal analysis}
We can choose to stay within $\glp_\omega$ and strengthen our base theory \theory{X}. So, if we wish to analyze some target theory \theory{U} with the $\Pi_1^0$ paradigm relative to \theory{X}, the question translates to how often one should iterate the Turing progression based on \theory{X} to comprise all the $\Pi^0_1$ consequences of \theory{U}. In the next section we shall analyze this in further detail.

\subsection{Beyond $\glp_\omega$}
Another choice to strengthen the applicability of the paradigm is to use modal provability logics that go beyond $\glp_\omega$. Currently most efforts of taking the paradigm further are along these lines. There are two main aspects involved here. The first is to extend the modal theory of \glp beyond $\glp_\omega$ and the other is to find suitable (hyper)arithmetical interpretations of the modalities $[\alpha]$ involved.

\subsubsection{The modal theory}
By now, the modal theory of $\glp_\Lambda$ is rather well studied and understood. A first and seminal step in this direction was taken by Beklemishev in \cite{Beklemishev:2005}. In particular, the paper focussed on the closed fragment $\glp^0$ of \glp and studied the worms in there. It was shown that  the orderings $<_0$ are well behaved also in the class-size $\glp^0$ and define a well order provided the irreflexivity of $<_0$. 

The irreflexivity of $<_0$ has been shown both in \cite{BeklemshevFernandezJoosten2011} and \cite{FernandezJoostenModels2012}. In particular \cite{FernandezJoostenModels2012} provides a class-size universal model for $\glp^0$. The ordering $<_0$ and natural and important generalizations are now well studied and understood as presented in \cite{Beklemishev:2005, FernandezJoostenWellOrders2012, FernandezJoostenWellFoundedPartialOrders2012}.

Although there are various important and interesting questions open in the modal theory of the logics $\glp_\Lambda$ it seems that all modal theory is in place to move the $\Pi^0_1$ ordinal analysis beyond \pa.

\subsubsection{Hyperarithmetic interpretations and the Reduction Property}

Currently the aim the \glp project is to provide an ordinal analysis of predicative analysis whose classical proof-theoretical ordinal is the Feferman-Sch\"utte ordinal $\Gamma_0$. Various natural candidates of provability notions have been seen to be sound and complete for $\glp_{\Gamma_0}$. However, so far, for none of this interpretations a natural generalization of the Reduction Property has been established.  

In the final section of this paper we shall briefly mention some of these generalized provability notions. In the next section we shall see how the need of a full Reduction Property can be circumvented.

\section{Reduction Property, equi-consistency, and relative ordinal analysis}

In this section we shall see how we can minimize the ingredients needed for a consistency proof as presented in Theorem \ref{theorem:PaIsConsistent}. In particular we shall not need the full Reduction Property but rather some weak version of it in terms of equi-consistency. 

We shall see that the following steps suffice. Below, let \theory{U} denote the target theory of which we wish to perform an ordinal analysis. 

\begin{enumerate}
\item
We fix some base theory \theory{X} over which most of our arguments will be performed;

\item\label{item:suitableConsistencyStatements}
We find some notions of consistency over \theory{X} of increasing strength 
\[
\{   \la 0 \ra_{\theory{X}} \varphi,    \la 1 \ra_{\theory{X}} \varphi,   \la 2 \ra_{\theory{X}} \varphi,   \la 3 \ra_{\theory{X}} \varphi, \ldots \},
\]
so that the following properties are obtained (we shall drop subscripts \theory{X})
\begin{enumerate}
\item\label{item:provableDeductionTheorem}
The notion $\la n \ra_{\theory{T}}$ grows monotone both in $n$ and in $\theory{T}$ and for all natural numbers $n$, theories \theory{T}, and formulas $\varphi, \psi$ we have that provably in some weak theory but certainly in \theory{X}
\[
\la n \ra_{\theory{T} + \varphi} \psi \ \leftrightarrow \ \la n \ra_{\theory{T} } (\psi \wedge  \varphi);
\]

\item\label{item:GLPisSound}
The logic \glp is sound for the corresponding dual provability operators $[n]_{\theory{X}}$;

\item\label{item:UincludedInConsistencies}
We have that (provably in some weak theory but certainly in \theory{X}) 
\[
\mathcal{U} \subseteq \theory{X} + \{   \la 0 \ra_{\theory{X}} \top,    \la 1 \ra_{\theory{X}} \top,   \la 2 \ra_{\theory{X}} \top,   \la 3 \ra_{\theory{X}} \top, \ldots \};
\]

\item\label{item:equiconsistencyReductionProperty}
The theory $\theory{X} + \la n+1 \ra \top$ is equi-consistent with the theory 
$\theory{X} + \{ Q^k_n(\top) \mid k \in \omega \}$ where $Q^0_n(\varphi)= \la n \ra \varphi$ and $Q^{k+1}_n (\varphi) = \la n \ra (\varphi \wedge Q^k_n(\varphi))$. This equi-consistency should be provable in some weak extension $\theory{X}^+$ of \theory{X}.

\end{enumerate}

\end{enumerate}

We shall now see that these ingredients suffice to perform a consistency proof of \theory{U} relative to \theory{X} formalized in $\theory{X}^+$.

\begin{theorem}\label{theorem:strippedConsistencyProof}
Suppose we have fixed \theory{X} and consistency notions as above. Then 
\[
\theory{X}^+ + {\sf TI}(\widetilde{\Pi^0_1}, \varepsilon_0) \vdash {\sf Con}(\theory{U}),
\]
where $\widetilde{\Pi^0_1}$ is some complexity class that corresponds to the consistency notion $\la 0 \ra_\theory{X}$.
\end{theorem}

\begin{proof}
The proof is similar to that of Theorem \ref{theorem:PaIsConsistent}. 
We reason in $\theory{X}^+$.
By \ref{item:UincludedInConsistencies} above, we have that $\theory{U} \subseteq \theory{X} + \{   \la 0 \ra_{\theory{X}} \top,    \la 1 \ra_{\theory{X}} \top,   \la 2 \ra_{\theory{X}} \top,   \la 3 \ra_{\theory{X}} \top, \ldots \}$ whence also
\[
\la 0 \ra_{\{   \la 0 \ra_{\theory{X}} \top,    \la 1 \ra_{\theory{X}} \top,   \la 2 \ra_{\theory{X}} \top,   \la 3 \ra_{\theory{X}} \top, \ldots \}}\top \to \la 0 \ra_\zfc \top.
\]
Clearly, by \ref{item:provableDeductionTheorem} we have that
\[
\la 0 \ra_{\{   \la 0 \ra_{\theory{X}} \top,    \la 1 \ra_{\theory{X}} \top,   \la 2 \ra_{\theory{X}} \top,   \la 3 \ra_{\theory{X}} \top, \ldots \}}\top \ \ \leftrightarrow \ \ \forall \, n \ \la 0 \ra \la n\ra \top.
\]
We now reason inside some weak extension $\theory{X}^+$ of \theory{X} and conclude by using transfinite induction and showing that $ \forall \, n \ \la 0 \ra_\theory{X} \la n\ra_\theory{X} \top$. Clearly it suffices to show that for all worms $A$ in $\glp_\omega$ we have that $\la 0\ra_\theory{X} A$. Thus, we set out to prove 
\begin{equation}\label{equation:progressive}
\forall  A\ [\forall \, B {<_0} A\  \la 0\ra_\theory{X} B \to \la 0\ra_\theory{X} A].
\end{equation}
We choose $\theory{X}^+$ strong enough so that it at least contains ${\sf RFN}_{\widetilde{\Sigma_1^0}}(\theory{X})$ in order to have 
\begin{enumerate}
\item
$\theory{X}^+ \vdash \la 0\ra_\theory{X} \top$ and,

\item
$\theory{X}^+ \vdash \la 0\ra_\theory{X} \varphi \ \to \ \la 0\ra_\theory{X} \la 0\ra_\theory{X} \varphi$.
\end{enumerate}
These two observations account for a proof of \eqref{equation:progressive} for the empty worm and worms of the form $\la 0 \ra A'$. For worms of the form $\la n+1 \ra A'$ we see by \ref{item:equiconsistencyReductionProperty} that 
\[
\la 0 \ra_\theory{X} \la n+1 \ra_\theory{X} A' \ \leftrightarrow \ \forall k \la 0 \ra_\theory{X} \la n \ra_\theory{X} Q^k_n(A').
\]
But, as by \ref{item:GLPisSound}, \glp is sound for our modalities, we get that all the $Q^k_n(A')$ are $<_0$-below $\la n+1 \ra \top$ and we have the right-hand side from the induction hypothesis.
\end{proof}
For the sake of presentation we have chosen \theory{X} and \theory{U} such that in some sense 
$\frac{\ea}{\pa} = \frac{\theory{X}}{\theory{U}}$ in which case we would be justified to say that the $\Pi^0_1$-proof theoretic ordinal of \theory{U} \emph{relative to} \theory{X} is $\varepsilon_0$. 

It is clear that Theorem \ref{theorem:strippedConsistencyProof} above can be extended to larger orderings once we have extended our notion of fundamental sequence as in Definition \ref{definition:FundamentalSequence} also for modalities with limit ordinals. This is unproblematic in principle but may slightly depend on the choice of fundamental sequences of the ordinals inside the modalities. The important observation is that the use of the full Reduction Property can be avoided. 

\section{Going beyond \pa: recent developments}

In this final section we just wish to briefly report on ongoing work to find arithmetical interpretations for $\glp_\Lambda$ with $\Lambda>\omega$.

\subsection{Truth-predicates and reflection}

Lev Beklemishev and Evgeniy Dashkov --both at Moscow State University-- have been studying interpretations of $\glp_{\omega\cdot 2 }$ where an additional truth-predicate for arithmetical formulas is added to the language of arithmetic. Within this framework they can express reflection for arithmetical formulas and slightly beyond. 

The work is still unpublished but they have presented some results where they go up to $[\omega + \omega]$ while preserving the full Reduction Property at the price of giving up the nice modal logic \glp. Rather they switch to a positive fragment of \glp to account for the fact that certain reflection principles (at limit stages) are not finitely axiomatizable.

\subsection{Omega rule interpretations}

Andr\'es Cord\'on Franco, David Fern\'andez Duque, and F\'elix Lara Mart\'{\i}n from the University of Sevilla, in collaboration with the author are studying an interpretation where, within an infinitary proof calculus $[\alpha]$ is read as ``provable with $\alpha$ nested applications of the omega rule". Soundness w.r.t.\ this interpretation has been proved and completeness seems feasible too. However, not much is known to what extend the Reduction Property holds for this interpretation.

\subsection{Levy's reflection results}

Joan Bagaria from ICREA and the University of Barcelona suggested in discussions with the author the following set-theoretical reading of our modalities $[n]$ for $n\in \omega$.
Let \theory{X} be the theory $\zfc - \{ {\sf Repl} + {\sf Inf} \}$. It is established in a paper from Levy (\cite{Levy:1961}) that 
\[
\zfc \equiv \theory{X} + \RFN(\theory{X}).
\]
Here, \RFN refers to the following notion of reflection: For each (externally quantified) natural number $n$, we denote by ${\sf RFN}_{\Sigma_n}(\theory{X})$ the following principle
\[
\forall \, \varphi{\in}\Sigma_n\, \forall a\, \exists\, \alpha {\in}{\sf On} \ [ V_\alpha \models \varphi (a) \ \Leftrightarrow \ \ \models_n \varphi (a) ].
\]
Here, $\models_n$ refers to partial truth predicates that are known to exist for \zfc and subtheories. At first sight it seems that replacement is needed to define the entities $V_\alpha$. However, in the absence of replacement one can work with the Scott-rank instead and define $V_\alpha : = \{  x\mid {\sf rank}(x)\leq \alpha \}$ where ${\sf rank}(x)\leq \alpha$ \emph{is} definable in \theory{X} making use of the transitive closure. We now define classes that collect the ordinals $\alpha$ for which the partial universes $V_\alpha$ are $\Sigma_n$ elementary substructures of $V$:
\[
C^{(n)} := \{ \alpha \mid V_\alpha \prec_{\Sigma_n} V\}.
\]
It is a theorem by Levy that the classes $C^{(n)}$ are $\Pi_n$ definable in \theory{X}.
Next, we define 
\[
\la n \ra_\theory{T} \varphi \ :\Leftrightarrow \ \exists\,  \alpha {\in} C^{(n)} \ [V_\alpha \models \theory{T} \wedge V_\alpha \models \varphi]
\]
It seems that all 2 (a)--(c) are satisfied for this notion of provability. In particular we have that $\la n \ra \varphi \to [m] \la n\ra \varphi$ since $\la n \ra_\theory{T}$ is definable in a $\Sigma_{n+1}$-fashion. As we cannot obtain that 
\[
\zfc \nvdash {\sf TI}({\Pi^0_1}, \varepsilon_0)
\]
we must conclude that $(d)$ does not hold and that the two theories are not equi-consistent.

\subsection{On a (relative) proof-theoretical ordinal of \zfc}

We conclude by a simple observation on a proof theoretical ordinal of \zfc. It is generally believed that an ordinal analysis for \zfc is currently way out of reach. With the methods presented here one might hope that at least an ordinal analysis relative to some strong base theory of \zfc might be possible. 

However, if such an analysis were to be given, it is most likely to be formalizable within \zfc itself. As \zfc proves transfinite induction over any well-ordering, this implies that  the order type involved in such an ordinal analysis of \zfc must be represented inside \zfc in such a way that \zfc does not prove it is indeed a well-order. 

\bibliographystyle{plain}
\bibliography{Biblio}

\begin{thebibliography}{10}

\bibitem{Levy:1961}
L\'evy A. and Vaught R.
\newblock Principles of partial reflection in the set theories of zermelo and
  ackermann.
\newblock {\em Pacific Journal of Mathematics}, 11:1045--1062, 1961.

\bibitem{BeklemshevFernandezJoosten2011}
L.~D. Beklemishev, D.~Fern\'andez-Duque, and J.~J. Joosten.
\newblock On transfinite provability logics, to appear in 2012.

\bibitem{Beklemishev:2004}
L.D. Beklemishev.
\newblock Provability algebras and proof-theoretic ordinals, {I}.
\newblock {\em Annals of Pure and Applied Logic}, 128:103--124, 2004.

\bibitem{BeklemishevSurvey:2005}
L.D. Beklemishev.
\newblock Reflection principles and provability algebras in formal arithmetic.
\newblock {\em Russian Mathematical Surveys}, 60:197--268, 2005.

\bibitem{Beklemishev:2005}
L.D. Beklemishev.
\newblock Veblen hierarchy in the context of provability algebras.
\newblock In {\em Logic, Methodology and Philosophy of Science, Proceedings of
  the Twelfth International Congress}. Kings College Publications, 2005.

\bibitem{Boo93}
G.~Boolos.
\newblock {\em The {L}ogic of {P}rovability}.
\newblock Cambridge University Press, Cambridge, 1993.

\bibitem{Japaridze:1988}
G.~Dzhaparidze.
\newblock The polymodal provability logic.
\newblock In {\em Intensional logics and the logical structure of theories:
  material from the {F}ourth {S}oviet-{F}innish Symposium on Logic}. Telavi,
  1988.

\bibitem{FernandezJoostenModels2012}
D.~Fern\'andez-Duque and J.~J. Joosten.
\newblock Models of transfinite provability logics, 2012.

\bibitem{FernandezJoostenWellFoundedPartialOrders2012}
D.~Fern\'andez-Duque and J.~J. Joosten.
\newblock Well-founded orders in the transfinite {J}aparidze algebra, 2012.

\bibitem{FernandezJoostenWellOrders2012}
D.~Fern\'andez-Duque and J.~J. Joosten.
\newblock Well orders in the transfinite {J}aparidze algebra, 2012.

\bibitem{Gentzen:1936}
G.~Gentzen.
\newblock Die wiederspruchsfreiheit der reinen zahlentheorie.
\newblock {\em Mathematische Annalen}, 112:493--565, 1936.

\bibitem{HP}
P.~H\'ajek and P.~Pudl\'ak.
\newblock {\em Metamathematics of {F}irst {O}rder {A}rithmetic}.
\newblock Springer-{V}erlag, Berlin, Heidelberg, New York, 1993.

\bibitem{Ignatiev:1993}
K.N. Ignatiev.
\newblock On strong provability predicates and the associated modal logics.
\newblock {\em The Journal of Symbolic Logic}, 58:249--290, 1993.

\bibitem{Leivant:1983}
D.~Leivant.
\newblock The optimality of induction as an axiomatization of arithmetic.
\newblock {\em Journal of Symbolic Logic}, 48:182--184, 1983.

\bibitem{Tait81}
W.~Tait.
\newblock Finitism.
\newblock {\em Journal of Philosophy}, 78:524--546, 1981.

\end{thebibliography}

\end{document}